\documentclass[12pt]{article}
\usepackage{amsmath,enumerate,amsfonts,amssymb,color,graphicx,amsthm}

\usepackage[colorlinks=true, allcolors=black]{hyperref}
\usepackage{todonotes}
\usepackage[normalem]{ulem}
\numberwithin{equation}{section}
\usepackage{cite}
\usepackage{mathtools}
\usepackage[nottoc,notlot,notlof]{tocbibind}
\usepackage[toc,page]{appendix}

\usepackage{lipsum}

\def\XXint#1#2#3{{\setbox0=\hbox{$#1{#2#3}{\int}$ }
		\vcenter{\hbox{$#2#3$ }}\kern-.6\wd0}}

\newlength{\leftstackrelawd}
\newlength{\leftstackrelbwd}
\def\leftstackrel#1#2{\settowidth{\leftstackrelawd}%
	{${{}^{#1}}$}\settowidth{\leftstackrelbwd}{$#2$}%
	\addtolength{\leftstackrelawd}{-\leftstackrelbwd}%
	\leavevmode\ifthenelse{\lengthtest{\leftstackrelawd>0pt}}%
	{\kern-.5\leftstackrelawd}{}\mathrel{\mathop{#2}\limits^{#1}}}

\theoremstyle{plain}
\newtheorem{thm}{Theorem}[section]
\newtheorem{lem}[thm]{Lemma}

\newtheorem{prop}[thm]{Proposition}
\newtheorem*{thm*}{Theorem}

\theoremstyle{definition}
\newtheorem{defn}[thm]{Definition}

\newtheorem{rmk}[thm]{Remark}
\newtheorem{?}[thm]{Problem}

\newcommand{\ring}{\mathring}

\newcommand{\ep}{\varepsilon}
\renewcommand{\phi}{\varphi}
\renewcommand{\epsilon}{\varepsilon}
\makeatletter
\def\@cite#1#2{[\textbf{#1\if@tempswa , #2\fi}]}
\def\@biblabel#1{[\textbf{#1}]}
\makeatother

\makeatletter
\newcommand*{\defeq}{\mathrel{\rlap{%
			\raisebox{0.3ex}{$\m@th\cdot$}}%
		\raisebox{-0.3ex}{$\m@th\cdot$}}%
	=}
\makeatother

\makeatletter
\newcommand*{\eqdef}{=\mathrel{\rlap{%
			\raisebox{0.3ex}{$\m@th\cdot$}}%
		\raisebox{-0.3ex}{$\m@th\cdot$}}%
	}
\makeatother

\newcounter{marnote}

\makeatletter
\def\underbracex#1#2{\mathop{\vtop{\m@th\ialign{##\crcr
				$\hfil\displaystyle{#2}\hfil$\crcr
				\noalign{\kern3\p@\nointerlineskip}%
				#1\crcr\noalign{\kern3\p@}}}}\limits}

\def\upbracefilla{$\m@th \setbox\z@\hbox{$\braceld$}%
	\bracelu\leaders\vrule \@height\ht\z@ \@depth\z@\hfill 
	\kern\p@\vrule \@width\p@\kern\p@\vrule \@width\p@\kern\p@\vrule \@width\p@
	$}

\def\upbracefillb{$\m@th \setbox\z@\hbox{$\braceld$}%
	\vrule \@width\p@\kern\p@\vrule \@width\p@\kern\p@\vrule \@width\p@\kern\p@
	\leaders\vrule \@height\ht\z@ \@depth\z@\hfill\bracerd
	\braceld\leaders\vrule \@height\ht\z@ \@depth\z@\hfill
	\kern\p@\vrule \@width\p@\kern\p@\vrule \@width\p@\kern\p@\vrule \@width\p@
	$}

\def\upbracefillc{$\m@th \setbox\z@\hbox{$\braceld$}%
	\vrule \@width\p@\kern\p@\vrule \@width\p@\kern\p@\vrule \@width\p@\kern\p@
	\leaders\vrule \@height\ht\z@ \@depth\z@\hfill
	\kern\p@\vrule \@width\p@\kern\p@\vrule \@width\p@\kern\p@\vrule \@width\p@
	$}

\def\upbracefilld{$\m@th \setbox\z@\hbox{$\braceld$}%
	\vrule \@width\p@\kern\p@\vrule \@width\p@\kern\p@\vrule \@width\p@\kern\p@
	\leaders\vrule \@height\ht\z@ \@depth\z@\hfill\braceru$}

\def\upbracefillbd{$\m@th \setbox\z@\hbox{$\braceld$}%
	\vrule \@width\p@\kern\p@\vrule \@width\p@\kern\p@\vrule \@width\p@\kern\p@
	\bracerd\braceld
	\leaders\vrule \@height\ht\z@ \@depth\z@\hfill\braceru$}

\makeatother

\begin{document}

\title{Properties of hypersurface singular sets of solutions to the $\sigma_k$-Yamabe equation in the negative cone}
\author{Jonah A. J. Duncan\footnote{Department of Mathematics, University College London, 25 Gordon Street, London, WC1H 0AY, UK. Email: jonah.duncan@ucl.ac.uk. Supported by the Additional Funding Programme for Mathematical Sciences, delivered by EPSRC (EP/V521917/1) and the Heilbronn Institute for Mathematical Research.} ~and Luc Nguyen\footnote{Mathematical Institute and St Edmund Hall, University of Oxford, Andrew Wiles Building, Radcliffe Observatory Quarter, Woodstock Road, OX2 6GG, UK. Email: luc.nguyen@maths.ox.ac.uk}}
\date{}
\maketitle
\vspace*{-5mm}

\begin{center} \textit{Dedicated to the memory of Professor Louis Nirenberg with admiration.}
\end{center}

\begin{abstract}
	We consider conformally flat Lipschitz viscosity solutions to the $\sigma_k$-Yamabe equation in the negative cone which admit smooth hypersurface singularities. Under natural regularity assumptions (that are satisfied by solutions to the $\sigma_k$-Loewner-Nirenberg problem on annuli, for example), we first prove that the trace and normal derivatives of such a solution along the hypersurface satisfy a certain PDE. For $k=2$, we also show that the hypersurface is minimal with respect to the Lipschitz solution and address some questions related to the formal expansion of the solution near the hypersurface. \bigskip 
	
	\end{abstract}

\setcounter{tocdepth}{1}
\tableofcontents

\newpage 

\section{Introduction}

Let $w$ be a positive $C^2$ function defined on an open set $\Omega\subset \mathbb{R}^n$ ($n\geq 3$). For the conformally flat Riemannian metric $g = w^{-2}|dx|^2$, define its Schouten tensor of type $(1,1)$ by
\begin{align*}
A_w = w\nabla^2 w - \frac{1}{2}|\nabla w|^2 I.
\end{align*}
Here, $\nabla^2 w$ denotes the Euclidean Hessian matrix of $w$ and $I$ denotes the $n\times n$ identity matrix. In this paper we consider continuous viscosity solutions to the $\sigma_k$-Yamabe equation in the so-called negative case:
\begin{align}\label{6}
\sigma_k(\lambda(-A_w)) = 1, \quad \lambda(-A_w)\in\Gamma_k^+ \quad \text{in }\Omega.
\end{align}
Here, $\lambda(-A_w)$ denotes the vector of eigenvalues of $-A_w$, $\sigma_k:\mathbb{R}^n\rightarrow\mathbb{R}$ denotes the $k$'th elementary symmetric polynomial and $\Gamma_k^+$ is the G{\aa}rding cone:
\begin{align*}
\sigma_k(\lambda_1,\dots,\lambda_n)  = \sum_{1\leq i_1 < \dots< i_k \leq n} \lambda_{i_1}\dots\lambda_{i_k}, \qquad \Gamma_k^+  = \{\lambda \in\mathbb{R}^n : \sigma_j(\lambda)>0 \text{ for }1 \leq j \leq k\}.
\end{align*}
By Li, Nguyen \& Wang \cite{LNW18}, solutions to \eqref{6} are locally Lipschitz in $\Omega$. The situation for the positive case, i.e. when $\lambda(-A_w)$ in \eqref{6} is replaced by $\lambda(A_w)$, is very different and is not considered here.\medskip

Our work is motivated by recent progress on the existence and regularity theory for solutions to the $\sigma_k$-Loewner-Nirenberg problem on $\Omega$:
\begin{align}\label{48}
\begin{cases}
\sigma_k(\lambda(-A_w)) = 1, \quad \lambda(-A_w)\in\Gamma_k^+ & \text{in }\Omega \\
w = 0 & \text{on }\partial\Omega. 
\end{cases}
\end{align}
The boundary condition in \eqref{48} comes from the requirement that the metric $g$ is complete in $\Omega$. \medskip

For $k=1$, existence, uniqueness and smoothness of solutions to \eqref{48} was established by Loewner \& Nirenberg in \cite{LN74}. For related results, including on Riemannian manifolds, see e.g.~\cite{AILA18, ACF92, Av82, AM88, Finn98, GW17, Gr17, HJS20, HS20, HN23, Jia21, Li22, Maz91, Ver81}. When $\Omega$ is a smooth bounded domain and $k\geq 2$, the existence of a Lipschitz viscosity solution to \eqref{48} and uniqueness in the class of continuous viscosity solutions was established by Gonz\'alez, Li \& Nguyen \cite{GLN18}. In \cite{LNX22}, Li, Nguyen \& Xiong showed that if $\partial\Omega$ has more than one connected component and $k\geq 2$, then the solution does not belong to $C^1_{\operatorname{loc}}(\Omega)$ but is smooth near $\partial \Omega$. Moreover, in the special case that $\Omega = \{a<|x|<b\}$ is an annulus, Li \& Nguyen \cite{LN20b} showed that the solution for $k\geq 2$ is smooth away from the hypersurface $\Sigma = \{|x| = \sqrt{ab}\}$, has a jump in the radial derivative across $\Sigma$, and is $C^{1,\frac{1}{k}}_{\operatorname{loc}}$ -- but not $C^{1,\gamma}_{\operatorname{loc}}$ for any $\gamma>\frac{1}{k}$ -- in each of $\{a<|x|\leq \sqrt{ab}\}$ and $\{\sqrt{ab} \leq |x| < b\}$. For related results, including on Riemannian manifolds, see e.g.~\cite{MP03, DN23, DN25a, DN25b, GSW11, Yuan22, Yuan24, CLL23, GG21, FW20, Wu24}. \medskip

Inspired by \cite{LN20b} and \cite{LNX22}, we consider in this paper the following type of singularity for a solution to \eqref{6}:

\begin{defn}\label{E}
	Let $w$ be a positive continuous function defined on a domain $\Omega\subset\mathbb{R}^n$, and suppose $B_\ep \Subset\Omega$. We say that $w$ has a hypersurface $\Sigma$ of singular points in $B_\ep$ if $\Sigma\subset\mathbb{R}^n$ is a smooth hypersurface and the following conditions hold:
	
	\begin{enumerate}
		\item $B_\ep\backslash \Sigma = B_\ep^+ \cup B_\ep^-$ where $B_\ep^\pm$ are non-empty, disjoint, open and connected (in particular $\Sigma\cap B_\ep\not=\emptyset$),
		\item $w \in C^{0,1}(\overline{B_\ep}) \cap C^1(\overline{B_\ep^\pm})\cap C^2_{\operatorname{loc}}(B_\ep^{\pm})$ but $w\not\in C^1(U)$ for any open set $U\subseteq B_\ep$ satisfying $U\cap \Sigma\not=\emptyset$,
		\item $w|_{\Sigma\cap B_\ep} \in C^2(\Sigma\cap B_\ep)$. 
	\end{enumerate}
\end{defn} 

The second condition in the above definition can be described as follows: we assume that $w$ is Lipschitz everywhere, $C^1$ on each side of $\Sigma$, $C^2$ away from $\Sigma$, but not $C^1$ across $\Sigma$ at any point. We point out that we do not assume $w\in C^{1,\delta}(\overline{B_\ep^\pm})$ for any $\delta>0$. \medskip

We fix a convention for the choice of $B_\ep^\pm $ as follows. Given a fixed point $x_0\in \Sigma \cap B_\ep$, we may assume after taking $\ep$ smaller if necessary that $\Sigma\cap B_\ep$ is given by the graph of a function $\phi$ defined over the tangent hyperplane to $\Sigma$ at $x_0$, and we set $B_\ep^\pm = \{(x',x_n)\in B_\ep: x_n \gtrless \phi(x')\}$. Here, $x'$ denotes the Cartesian coordinates on the tangent hyperplane $\{x_n=0\}$, and in these coordinates $x_0 = (0',0)$. We denote by $\nu$ the unit normal vector of $\Sigma$ pointing into $B_\ep^+$ and denote $w^\pm = w|_{B_\ep^\pm}$.\medskip

Our first main result shows that if a viscosity solution $w$ to \eqref{6} has a hypersurface $\Sigma$ of singular points in $B_\ep$, then a certain PDE is satisfied on $\Sigma\cap B_\ep$ by $w_0 \defeq w|_{\Sigma\cap B_\ep}$ and $\nabla_\nu w^\pm$. In what follows, for $\alpha\in\mathbb{R}$ we denote
\begin{align}\label{71}
T_\alpha \defeq -w_0(\nabla_\Sigma^2 w_0 - \alpha\mathrm{I\!I}_\Sigma) + \frac{1}{2}\big(|\nabla_\Sigma w_0|^2 + \alpha^2\big)g_\Sigma
\end{align}
where $g_\Sigma$ is the induced metric on $\Sigma$ with respect to the ambient Euclidean metric, $\mathrm{I\!I}_\Sigma$ is the second fundamental form of $\Sigma$ with respect to the ambient Euclidean metric, and a subscript $\Sigma$ on a differential operator means it is defined with respect to the induced metric $g_\Sigma$.

\begin{thm}\label{A}
	Suppose that $w$ is a viscosity solution to \eqref{6} admitting a hypersurface $\Sigma$ of singular points in $B_\ep$. Then for $\alpha \in \{\nabla_\nu w^\pm\}$,
	\begin{align}\label{26}
	\lambda\big(g_\Sigma^{-1}T_\alpha\big)\in\partial\Gamma_{k-1}^+ \quad \text{on }\Sigma\cap B_\ep.
	\end{align}
\end{thm}

The equation \eqref{26} is a polynomial equation of order $2k-2$ in $\alpha$.\medskip 

Roughly speaking, the equation \eqref{26} indicates strongly that the failure of $C^1$ regularity across $\Sigma$ is not arbitrary. It forces a precise nonlinear partial differential relation between the trace and normal derivatives across $\Sigma$.\medskip 

For a partial result when $\Gamma_k$ is replaced by a general cone $\Gamma$, see Theorem \ref{7}.\medskip

For the remainder of the introduction we consider only the case $k=2$, in which case \eqref{26} reads
\begin{align}\label{32}
-w_0\Delta_\Sigma w_0 + \alpha w_0H_\Sigma+ \frac{n-1}{2}|\nabla_\Sigma w_0|^2 + \frac{n-1}{2}\alpha^2 = 0 \quad \text{on }B_\ep\cap \Sigma,
\end{align}
where $H_\Sigma = \sigma_1(g_\Sigma^{-1}\mathrm{I\!I}_\Sigma)$ is the mean curvature of $\Sigma$ with respect to $\nu$.\medskip 

Our second main result is as follows:  
\begin{thm}\label{D}
		Suppose that $k=2$ and $w$ is a viscosity solution to \eqref{6} admitting a hypersurface $\Sigma$ of singular points in $B_\ep$. Then the mean curvature of $\Sigma\cap B_\ep$ is negative with respect to $g|_{\overline{B_\ep^\pm}} = w_\pm^{-2}|dx|^2$ and $\pm \nu$. In particular, $\Sigma\cap B_\ep$ is minimal in the variational sense, i.e.~it is stationary for the area functional associated to the Lipschitz metric $g|_{B_\ep}$, with respect to compactly supported variations of $\Sigma\cap B_\ep$.
\end{thm}

\begin{rmk}
By the transformation law \eqref{74} below, $w \in C^1(\overline{B_\ep^\pm})$ implies that the mean curvatures of $B_\ep\cap \Sigma$ with respect to $g|_{\overline{B_\ep^\pm}}$ are well-defined.
\end{rmk}

\begin{rmk}
\begin{enumerate}[(i)]
\item In \cite{LNX22}, the theory of minimal hypersurfaces was used to
prove the non-existence of $C^1$ solutions to the $\sigma_k$-Loewner-Nirenberg problem when $k \geq 2$ in domains with more than one boundary component. Two important steps of \cite{LNX22} are
\begin{itemize}
\item The asymptotic behaviour of a solution $w(x) \sim d(x,\partial\Omega)$ near $\partial\Omega$ implies that the hypersurfaces parallel to $\partial\Omega$ are mean-convex with respect to $g = w^{-2}|dx|^2$ and the normal pointing into $\Omega$. Therefore, when there is more than one boundary component, there exists a minimal hypersurface (with respect to $g$) confined to $\Omega$ with possibly lower dimensional singularities. 
\item There is a differential inequality for any $w\in C^1$ satisfying $\lambda(-A_w)\in\overline{\Gamma_2^+}$ (see Corollary 3.9 therein), which essentially acts as an obstruction to the existence of a confined minimal hypersurface with respect to $g$, yielding a contradiction. Crucially, this inequality is not valid when $w$ is only Lipschitz.
\end{itemize}

The assertion in Theorem \ref{D} that $\Sigma\cap B_\ep$ is minimal with respect to $g|_{B_\ep}$ for $k = 2$ therefore directly echos the above picture.

\item It would be interesting to understand the geometric nature of the singular set $\Sigma$ for $k\geq 3$. 
\end{enumerate}
\end{rmk}

One may verify directly that the assertion of Theorem \ref{D} holds for solutions to the $\sigma_k$-Loewner-Nirenberg problem on annuli for all $k\geq 2$ as follows. In \cite[Theorem 1.3]{LN20b}, Li \& Nguyen showed that on the annulus $\Omega = \{a<|x|<b\}$, the viscosity solution $u^{\frac{4}{n-2}}|dx|^2$ to the $\sigma_k$-Loewner-Nirenberg problem has a jump in the radial derivative at $|x| = \sqrt{ab}$, and with the choice $\Omega^+ = \{|x|>\sqrt{ab}\}$ (so that $\nu$ is the outward pointing unit normal), it holds that
\begin{align*}
\nabla_\nu \ln u^- = -\frac{n-2}{\sqrt{ab}} \quad \text{and} \quad \nabla_\nu \ln u^+ = 0. 
\end{align*}
Now recall that for $g = u^{\frac{4}{n-2}}|dx|^2$, one has the following transformation law for the mean curvature of a hypersurface $\Sigma$:
 \begin{align}\label{74}
H_{(\Sigma, g)} = u^{-\frac{2}{n-2}}\bigg(-\frac{2(n-1)}{n-2}\nabla_\nu \ln u + H_\Sigma \bigg).
\end{align}
Taking $\Sigma = \{|x|=\sqrt{ab}\}$ and noting that $H_\Sigma = -\frac{n-1}{\sqrt{ab}}$, we see that with respect to $\nu$,
\begin{align*}
H_{(\Sigma, \, (u^-)^{\frac{4}{n-2}}|dx|^2)} = u^{-\frac{2}{n-2}}\frac{n-1}{\sqrt{ab}}>0 \quad \text{and} \quad H_{(\Sigma, \, (u^+)^{\frac{4}{n-2}}|dx|^2)} = -u^{-\frac{2}{n-2}}\frac{n-1}{\sqrt{ab}}<0,
\end{align*}
as claimed.\medskip

As pointed out earlier, the solution to the $\sigma_k$-Loewner-Nirenberg problem on the annulus $\Omega = \{a<|x|<b\}$ is $C^{1,\frac{1}{k}}_{\operatorname{loc}}$ but not $C^{1,\gamma}_{\operatorname{loc}}$ for any $\gamma>\frac{1}{k}$ in each of $\{a<|x|\leq \sqrt{ab}\}$ and $\{\sqrt{ab} \leq |x| < b\}$. We expect such H\"older regularity of the gradient to hold in more general domains when the solution admits a hypersurface of singular points. A key difficulty in this more general setting (where ODE methods are not obviously applicable) is that one expects the equation to become degenerate elliptic at the singular hypersurface, which complicates the asymptotic analysis. (The degenerate ellipticity can be derived from the computations in Section \ref{s3} although we do not do this explicitly.)  This is in contrast with the behaviour of the solution to the $\sigma_k$-Loewner-Nirenberg problem near the boundary of the domain, where the equation remains uniformly elliptic (since the solution is asymptotically hyperbolic) and $C^2$ estimates are available. This fact was used by Li, Nguyen \& Xiong \cite{LNX22} in combination with the small perturbation argument of Savin \cite{Sav07} to obtain polyhomogeneous expansions at the boundary to any order.\medskip

Nonetheless, we provide some evidence in support of the expected H\"older regularity of the gradient, at least for $k=2$, in the form of the following proposition (we restrict our attention to $B_\ep^+$; the amendments in $B_\ep^-$ are obvious). We will assume that $\widetilde{w}_0>0$ and $w_1$ are two functions defined on $\Sigma\cap B_\ep^+ $ satisfying the conclusions of Theorems \ref{A} and \ref{D}, that is
\begin{align}\label{88}
-\widetilde{w}_0\Delta_\Sigma \widetilde{w}_0 +  \widetilde{w}_0w_1H_\Sigma+ \frac{n-1}{2}|\nabla_\Sigma \widetilde{w}_0|^2 + \frac{n-1}{2}w_1^2 = 0
\end{align}
and
\begin{align}\label{65'}
\widetilde{w}_0H_\Sigma + (n-1)w_1 < 0 \quad \text{on }\Sigma\cap B_\ep^+.
\end{align}
Note that when $\widetilde{w}_0=w_0$ and $w_1 = \nabla_\nu w^+$, the LHS of \eqref{65'} is precisely the mean curvature of $\Sigma$ with respect to $g|_{\overline{B_\ep^+}} = w_+^{-2}|dx|^2$ and $\nu$, which can be seen by making the substitution $u = w^{-\frac{n-2}{2}}$ in the more familiar transformation law \eqref{74}. \medskip

Suppose that $\varepsilon$ is sufficiently small such that every point in $B_\varepsilon$ has a unique closest point on $\Sigma$. Let $\pi$ denote the closest point projection onto $\Sigma$. We search for a function $w_*^+ \not \equiv 0$ defined on $\Sigma$ and $p\in (1,2)$ such that
	\begin{align}
	\text{the function } &\bar{w}^+(x) \defeq \widetilde{w}_0(\pi(x)) +  w_1 (\pi(x)) \mathrm{d}(x) + w_*^{+}(\pi(x)) \mathrm{d}(x)^{p} \text{ satisfies }\nonumber\\
	&
	\begin{cases}
	\lambda(-A_{\bar{w}^+})\in\Gamma_2^+ \text{ in }B_\ep^+,\\ 
	\lim_{d\rightarrow 0} \sigma_2(\lambda(-A_{\bar{w}^+})) = 1.
	\end{cases}
	\label{69}
	\end{align}

\begin{prop}\label{B}
	Suppose $\widetilde{w}_0>0$ and $w_1$ satisfy \eqref{88} and \eqref{65'} on $\Sigma\cap B_\ep^+$. Then there exists a function $w_*^+ \not \equiv 0$ defined on $\Sigma$  and $p\in (1,2)$ such that \eqref{69} holds after possibly shrinking $\varepsilon$ if and only if $p=\frac{3}{2}$. Moreover, when $p = \frac{3}{2}$, $w_*^+ < 0$ and is uniquely and explicitly determined by $\tilde w_0, w_1$ and $\Sigma$. An analogous statements also hold in $B_\ep^-$. 
\end{prop}

We note that the exponent $p = \frac{3}{2}$ is consistent with the observed $C^{1,\frac{1}{2}}$-regularity in \cite{LN20b} for the $\sigma_2$-Loewner-Nirenberg problem on annuli.\medskip

The plan of the paper is as follows. In Section \ref{s5} we prove Theorems \ref{A} and \ref{D}, and in Section 
\ref{s3} we prove Proposition \ref{B}.

\subsubsection*{Rights retention statement} For the purpose of Open Access, the authors have applied a CC BY public copyright licence to any Author Accepted Manuscript (AAM) version arising from this submission.

\section{Proof of Theorems \ref{A} and \ref{D}}\label{s5}

In this section we prove Theorems \ref{A} and \ref{D}, which concern the equation satisfied on the singular set $\Sigma$ by $w|_\Sigma$ and $\nabla_\nu w^\pm$, and the minimality of $\Sigma$ when $k=2$. We begin in Section \ref{s1} with a preliminary lemma (Lemma \ref{4}) which provides an ordering of the normal derivatives $\nabla_\nu w^{\pm}$. We also prove a related result (Theorem \ref{7}) whose proof follows a similar argument to that of Lemma \ref{4} (we note that Theorem \ref{7} will not be used in the rest of the paper). In Section \ref{s2} we give the proof of Theorem \ref{A}, and in Section \ref{s4} we give the proof of Theorem \ref{D}. \medskip

We will work at a fixed point $x_0\in \Sigma\cap B_\ep$ (or a neighbourhood thereof), which we take to be the origin of the coordinate system described in the introduction. We locally define
\begin{align*}
\mathring{w}(x') = w(x', \phi(x'))
\end{align*}
(so that in particular $\mathring{w}(0') = w(0', 0)$), and for $\alpha,\beta\in\mathbb{R}$ we locally define
\begin{align}\label{52}
\xi_{\alpha,\beta}(x', x_n) = \ring{w}(x') + \alpha(x_n - \phi(x')) + \frac{1}{2}\beta(x_n - \phi(x'))^2
\end{align}
and
\begin{align}\label{17}
(T_\alpha(0'))_{ab} = -\ring{w}(0')\big(\partial_{ab} \ring{w}(0') - \alpha \partial_{ab} \phi(0')\big) + \frac{1}{2}\big(|\nabla_\Sigma \ring{w}(0')|^2 + \alpha^2\big)\delta_{ab}
\end{align}
for $1 \leq a, b, \leq n-1$. Note that \eqref{17} is the expression in our local coordinates at $0'$ of the tensor $T_\alpha$ defined in \eqref{71}. Recalling also that $\nu$ is the unit normal vector of $\Sigma$ pointing into $B_\ep^+$, we see that this coincides at $0'$ with the $n$'th coordinate vector $e_n$.

\subsection{A preliminary lemma on the ordering of $\nabla_\nu w^\pm$}\label{s1}

The results in this subsection apply more generally to positive viscosity supersolutions of the equation
\begin{align}\label{1}
\lambda(-A_w)\in\partial\Gamma \quad \text{in }B_\ep,
\end{align}
i.e.~positive functions $w$ satisfying $\lambda(-A_w)\in\overline{\Gamma}$ in the viscosity sense, where $\Gamma$ satisfies the following standard assumptions: 
\begin{align}
& \Gamma\subset\mathbb{R}^n\text{ is an open, convex, connected symmetric cone with vertex at }0, \label{72} \\
& \Gamma_n^+ = \{\lambda\in\mathbb{R}^n:  \lambda_i > 0 ~\forall ~1\leq i \leq n\}  \subseteq \Gamma  \subseteq \Gamma_1^+ =  \{\lambda \in\mathbb{R}^n: \lambda_1+\dots+\lambda_n > 0\}. \label{73} 
\end{align}

\begin{lem}\label{4}
	Suppose $\Gamma$ satisfies \eqref{72} and \eqref{73}, and let $w$ be a positive viscosity supersolution to \eqref{1} admitting a hypersurface $\Sigma$ of singular points in $B_\ep$. Then 
	\begin{align*}
	 \nabla_\nu w^+ < \nabla_\nu w^- \quad\text{on }\Sigma \cap B_\ep.
	\end{align*}
\end{lem}

\begin{proof}
	Let $x_0$ and the associated coordinate system be as described previously. To prove the lemma, it is equivalent to show 
	\begin{align*}
	\partial_n w^+(0) < \partial_n w^-(0),
	\end{align*}
	where we write $0$ for the point $x_0 = (0', 0)$. We know by our choice of $x_0$ that $\partial_n w^+(0) \not= \partial_n w^-(0)$. Let us assume for a contradiction that $\partial_n w^-(0) < \partial_n w^+(0)$, and take $\alpha\in\mathbb{R}$ such that 
	\begin{align}\label{2}
	\partial_n w^-(0) < \alpha < \partial_n w^+(0). 
	\end{align}
	
	Recall the definition of $\xi_{\alpha,\beta}$ from \eqref{52}, and note that this belongs to $C^2(B_\ep)$ by the third condition in Definition \ref{E}. We first observe that
	\begin{align}\label{53}
	\xi_{\alpha,\beta}(x',x_n) - w(x', x_n) & = \alpha (x_n - \phi(x')) + w(x',\phi(x')) - w(x', x_n) + \frac{1}{2}\beta(x_n - \phi(x'))^2  \nonumber \\
	& = \Big[\alpha - \partial_n w(x', \eta) + \frac{1}{2}\beta(x_n - \phi(x'))\Big](x_n - \phi(x'))
	\end{align}
	for some $\eta$ between $x_n$ and $\phi(x')$. It follows from \eqref{2} and \eqref{53} that, after taking $\ep$ smaller if necessary, 
	\begin{align*}
	\xi_{\alpha,\beta} >0 \quad \text{and} \quad \xi_{\alpha,\beta} \leq w \quad \text{in }B_\ep. 
	\end{align*}
	Since $\xi_{\alpha,\beta}(0) = w(0)$, $\xi_{\alpha,\beta}$ is therefore an admissible test function at $0$ in the definition of viscosity supersolution, yielding
	\begin{align}\label{3}
	\lambda(-A_{\xi_{\alpha,\beta}}(0))\in\overline{\Gamma}
	\end{align}
	for any $\beta\in\mathbb{R}$. \medskip 
	
	For $a,b\in\{1,\dots,n-1\}$, we now compute
	\begin{align}\label{11}
	\partial_a \xi_{\alpha,\beta} & = \partial_a \ring{w}- \alpha \partial_a \phi - \beta(x_n - \phi)\partial_a \phi \nonumber \\
	\partial_n \xi_{\alpha,\beta} & = \alpha + \beta (x_n - \phi) \nonumber \\
	\partial_{ab} \xi_{\alpha,\beta} & = \partial_{ab} \ring{w} - \alpha \partial_{ab} \phi + \beta \partial_a \phi \partial_b \phi - \beta(x_n - \phi)\partial_{ab}\phi \nonumber \\
	\partial_{an} \xi_{\alpha,\beta} & = -\beta \partial_a \phi \nonumber \\
	\partial_{nn}\xi_{\alpha,\beta} & = \beta.
	\end{align}
	It follows that
	\begin{align}\label{12}
	\nabla^2 \xi_{\alpha,\beta}(0) = \left(
	\begin{array}{c|c}
	\partial_{ab}w(0) - \alpha \partial_{ab} \phi(0) & 0\\
	\hline
	0 & \beta
	\end{array}
	\right),
	\end{align}
	and moreover the values $\xi_{\alpha,\beta}(0) = w(0)$ and $|\nabla \xi_{\alpha,\beta}(0)|^2$ are independent of $\beta$. Therefore, for some matrix $M$ independent of $\beta$, \eqref{3} can be rewritten as
	\begin{align*}
	\lambda(-\beta w(0) e_n\otimes e_n + M)\in\overline{\Gamma}. 
	\end{align*}
	In particular, $\operatorname{tr}(-\beta w(0) e_n\otimes e_n + M) = -\beta w(0) +\operatorname{tr}(M) \geq 0$, which is a contradiction for $\beta \gg 0$. 
\end{proof}

Similar ideas lead to the following: 

\begin{thm}\label{7}
	Suppose $\Gamma$ satisfies $(1,0,\dots,0)\in\Gamma$, \eqref{72} and \eqref{73}, and let $w$ be a viscosity supersolution to $\lambda(-A_w)\in\partial\Gamma$ on some domain $\Omega\subset\mathbb{R}^n$. Then $w$ does not admit a hypersurface of singular points in any $B_\ep \Subset \Omega$. 
\end{thm}

\begin{rmk}
	Under the assumption $(1,0,\dots,0)\in\Gamma$, the equation $\lambda(-A_w)\in\partial\Gamma$ is locally strictly elliptic -- see \cite[Proposition A.1]{LN20}. 
\end{rmk}
 
\begin{rmk}
When $(1, 0, \ldots, 0) \in \Gamma$, Theorem \ref{7} implies that solutions to \eqref{6} with $\Gamma_k$ replaced by $\Gamma$ do not admit a hypersurface of singular points. This is consistent with the fact that a priori second derivative estimates hold in such case (see \cite{DN23, Guan08, GV03b}).
\end{rmk}

\begin{proof}[Proof of Theorem \ref{7}]
	We suppose for a contradiction that $w$ admits a hypersurface of singular points in some $B_\ep\Subset \Omega$, so that by Lemma \ref{4} one can take $\alpha\in\mathbb{R}$ satisfying
	\begin{align*}
	\partial_n w^+(0)<\alpha<\partial_n w^-(0).
	\end{align*}
	Using similar arguments to the proof of Lemma \ref{4}, we see 
	\begin{align*}
	\xi_{\alpha,\beta} > 0  \quad \text{and} \quad \xi_{\alpha,\beta} \geq w  \quad \text{in }B_\ep. 
	\end{align*}
	Since $\xi_{\alpha,\beta}(0) = w(0)$, $\xi_{\alpha,\beta}$ is therefore an admissible test function at 0 in the definition of viscosity subsolution, yielding
	\begin{align*}
	\lambda(-A_{\xi_{\alpha,\beta}}(0))\in\mathbb{R}^n\backslash\Gamma
	\end{align*}
	for any $\beta\in\mathbb{R}$. The same computation as before then tells us that for some matrix $M$ independent of $\beta$ we have 
	\begin{align*}
	\lambda(-\beta w(0)e_n\otimes e_n + M)\in\mathbb{R}^n\backslash\Gamma,
	\end{align*}
	which is a contradiction for $\beta\ll0$ since $(1,0,\dots,0)\in\Gamma$. 
\end{proof}

\subsection{Proof of Theorem \ref{A}}\label{s2}

We first prove one side of the differential inclusion claimed in \eqref{26}, which only requires $\nabla_\nu w^+ \leq \alpha \leq \nabla_\nu w^-$ (rather than $\alpha =  \nabla_\nu w^\pm$):

\begin{lem}\label{24}
	Suppose $w$ is a viscosity solution to \eqref{6} admitting a hypersurface of singular points in $B_\ep$. Then 
	\begin{align*}
	\lambda(g_\Sigma^{-1}T_\alpha)\in\mathbb{R}^n\backslash\Gamma_{k-1}^+ \quad \text{in }\Sigma \cap B_\ep \quad \text{for }  \nabla_\nu w^+ \leq \alpha \leq \nabla_\nu w^-.
	\end{align*}
\end{lem}

\begin{proof}
	Again we let $x_0$ and the associated coordinate system be as described previously. To prove the lemma, it is equivalent to show $\lambda((T_\alpha(0'))_{ab})\in\mathbb{R}^n\backslash\Gamma_k^+$ for $\partial_n w^+(0) \leq \alpha \leq \partial_n w^-(0)$.\medskip 
	
	First observe that \eqref{11} and \eqref{12} combine to give
	\begin{align}\label{8}
	-A_{\xi_{\alpha,\beta}}(0) = \left(
	\begin{array}{c|c}
	(T_\alpha(0'))_{ab} & 0\\
	\hline
	0 & -\beta w(0) + \frac{1}{2}(|\nabla_\Sigma w_0(0)|^2 + \alpha^2)
	\end{array}
	\right).
	\end{align}
	If $\partial_n w^+(0) < \alpha < \partial_n w^-(0)$ then, as in the proof of Theorem \ref{7}, $\xi_{\alpha,\beta}$ is an admissible test function at $0$ in the definition of viscosity subsolution to \eqref{6}. Therefore either $\lambda(-A_{\xi_{\alpha,\beta}}(0))\in\mathbb{R}^n\backslash\Gamma_k^+$ or 
	\begin{align*}
	\sigma_k(\lambda(-A_{\xi_{\alpha,\beta}}(0))) \leq 1, \quad \lambda(-A_{\xi_{\alpha,\beta}}(0))\in\Gamma_k^+. 
	\end{align*}
	By continuity, the same dichotomy holds for $\partial_n w^+(0) \leq \alpha \leq \partial_n w^-(0)$. Now if it were the case that $\lambda((T_\alpha(0'))_{ab})\in\Gamma_{k-1}^+$, then for $\beta \ll 0$, \eqref{8} would imply $\lambda(-A_{\xi_{\alpha,\beta}}(0))\in\Gamma_k^+$ and $\sigma_k(\lambda(-A_{\xi_{\alpha,\beta}}(0)))>1$, contradicting the dichotomy. Therefore $\lambda((T_\alpha(0'))_{ab})\in\mathbb{R}^n\backslash\Gamma_k^+$ for $\partial_n w^+(0) \leq \alpha \leq \partial_n w^-(0)$.
\end{proof}

We now address the other side of the differential inclusion in \eqref{26}, which requires $\alpha=\nabla_\nu w^\pm$:

\begin{lem}\label{21}
	Suppose $w$ is a viscosity solution to \eqref{6} admitting a hypersurface of singular points $B_\ep$. Then 
	\begin{align}\label{60}
	\lambda(g_\Sigma^{-1}T_\alpha)\in \overline{\Gamma_{k-1}^+} \quad \text{on } \Sigma \cap B_\ep \quad \text{for }  \alpha=\nabla_\nu w^\pm.
	\end{align}
\end{lem}

\begin{proof}[Proof of Lemma \ref{21}]

Again we let $x_0$ and the associated coordinate system be as described previously. To prove the lemma, it is equivalent to show $\lambda((T_\alpha(0'))_{ab})\in\overline{\Gamma_k^+}$ for $\alpha = \partial_n w^\pm (0)$. We consider only the case $\alpha = \partial_n w^+(0)$, since the other case is essentially identical. In what follows, we will first take $\partial_n w^+(0) < \alpha < \partial_n w^-(0)$, which is possible by Lemma \ref{4}, and later in the proof we will take $\alpha\rightarrow \partial_n w^+(0)$. \medskip 

For parameters $0<\ep_2<\ep_1<1$, $0<\gamma<1$ and $\beta\in\mathbb{R}$, we define
\begin{align*}
\xi_{\alpha,\beta, \gamma}(x', x_n) = \ring{w}(x') - \gamma|x'|^2 + \alpha(x_n - \phi(x')) + \frac{1}{2}\beta(x_n - \phi(x'))^2 
\end{align*}
and
\begin{align*}
\Omega_{\ep_1,\ep_2} = \{(x',x_n): |x'|<\ep_1 \text{ and } 0<x_n - \phi(x')<\ep_2\}. 
\end{align*}
Then
\begin{align*}
\xi_{\alpha,\beta,\gamma}  - w = \big[\alpha - \partial_n w(x',\eta)\big](x_n-\phi(x')) - \gamma|x'|^2 + \frac{1}{2}\beta(x_n - \phi(x'))^2
\end{align*}
for some $\eta$ between $x_n$ and $\phi(x')$. \medskip 

The first main ingredient in the proof is the following claim, wherein we denote
\begin{align}\label{64}
\delta \defeq |\alpha - \partial_n w^+(0)| + |\operatorname{osc}_{|x'|<\ep_1,\, 0<x_n - \phi(x') < \ep_1} \partial_n w^+|. 
\end{align}

\noindent\textit{Claim 1:} Suppose $\ep_1$ and $|\alpha - \partial_n w^+(0)|$ are sufficiently small so that $\delta<1$, and let
\begin{align}\label{22}
\ep_2 = \gamma\ep_1^2
\end{align}
and 
\begin{align}\label{18}
\beta = -\frac{3\delta}{\ep_2}<0.
\end{align}
Then $\xi_{\alpha,\beta,\gamma}\leq w$ on $\partial\Omega_{\ep_1,\ep_2}$ with equality if and only if $x'=0$.\medskip 

\noindent\textit{Proof of Claim 1. } We first consider the portion of the boundary given by $\partial \Omega_{\ep_1,\ep_2}\cap \{|x'|=\ep_1\}$. On this set, since $\beta<0$ we have
\begin{align*}
\xi_{\alpha,\beta,\gamma} - w &  \leq \big|\alpha - \partial_n w(x',\eta)\big|(x_n - \phi(x')) - \gamma\ep_1^2 \nonumber \\
& \leq \delta\ep_2 - \gamma \ep_1^2
\end{align*}
for $\delta$ defined as above. By the choice of $\ep_2$ in \eqref{22} and the fact that $\delta<1$, we therefore obtain the inequality $\xi_{\alpha,\beta,\gamma} - w <0$ on $\partial \Omega_{\ep_1,\ep_2}\cap \{|x'|=\ep_1\}$.\medskip 

Next we consider the portion of the boundary given by $\partial \Omega_{\ep_1,\ep_2} \cap \{x_n - \phi(x') = \ep_2\}$. On this set we have
\begin{align*}
\xi_{\alpha,\beta,\gamma} - w & \leq \delta\ep_2 + \frac{1}{2}\beta\ep_2^2 \nonumber \\
& = \frac{1}{2}\ep_2(2\delta + \beta\ep_2),
\end{align*}
and hence by the choice of $\beta$ in \eqref{18} we obtain the inequality $\xi_{\alpha,\beta,\gamma} - w <0$ on $\partial \Omega_{\ep_1,\ep_2} \cap \{x_n - \phi(x') = \ep_2\}$. \medskip 

The remaining part of the boundary to consider is $\partial \Omega_{\ep_1,\ep_2} \cap \{x_n = \phi(x')\}$. But on this set it is clear that $\xi_{\alpha,\beta, \gamma}\leq w$ with equality if and only if $x' = 0$. This completes the proof of Claim 1. \medskip

\noindent\textit{Claim 2:} There exists $0<c \leq  \delta\ep_2 $ and an interior point $x_1\in\Omega_{\ep_1,\ep_2}$ such that
\begin{align*}
\sigma_k(\lambda(-A_{\xi_{\alpha,\beta, \gamma} - c}(x_1))) \geq 1, \quad \lambda(-A_{\xi_{\alpha,\beta, \gamma} - c}(x_1)) \in\Gamma_k^+
\end{align*}
where 
\begin{align*}
- A_{\xi_{\alpha,\beta, \gamma} - c}(x_1) =  -(\xi_{\alpha,\beta, \gamma} - c)(x_1)\nabla^2 \xi_{\alpha,\beta, \gamma}(x_1) + \frac{1}{2}|\nabla\xi_{\alpha,\beta, \gamma}|^2(x_1)I.
\end{align*}

\noindent\textit{Proof of Claim 2.} On the line $\{x' = 0\}$, we have for some $\eta$ between 0 and $x_n$
\begin{align*}
(\xi_{\alpha,\beta, \gamma} - w)(0',x_n) & = \big[\alpha - \partial_n w^+(0',\eta)\big]x_n + \frac{1}{2}\beta x_n^2 \nonumber \\
& \geq \big[\alpha - \partial_n w^+(0) - \operatorname{osc}_{0\leq t \leq x_n} \partial_n w^+(0',t)\big]x_n + \frac{1}{2}\beta x_n^2,
\end{align*}
which is strictly positive for all $x_n>0$ sufficiently small since we assume $\alpha>\partial_n w^+(0)$. On the other hand, $\xi_{\alpha,\beta, \gamma}-w$ cannot be too large in $\Omega_{\ep_1,\ep_2}$, since 
\begin{align*}
\xi_{\alpha,\beta, \gamma} - w & \leq \big|\big(\alpha - \partial_n w^+(x',\eta)\big)(x_n - \phi(x'))\big|  \nonumber \\
& \leq \delta\ep_2 . 
\end{align*}
Combining these two facts and Claim 1 (which tells us $\xi_{\alpha,\beta, \gamma} - w \leq 0$ on $\partial\Omega_{\ep_1,\ep_2}$), we see there exists $0<c \leq  \delta\ep_2 $ and an interior point $x_1\in\Omega_{\ep_1,\ep_2}$ such that
\begin{align*}
\xi_{\alpha,\beta, \gamma} - c \leq w \quad \text{in }\Omega_{\ep_1,\ep_2} \quad \text{and} \quad (\xi_{\alpha,\beta, \gamma}-c)(x_1) = w(x_1). 
\end{align*}
Claim 2 then follows immediately from the definition of a viscosity supersolution.\medskip

Roughly speaking, we would now like to take $\ep_1\rightarrow 0$, $\alpha\rightarrow \partial_n w^+(0)$ and $\gamma\rightarrow 0$ to conclude the proof of the lemma. This requires a careful error term analysis. In what follows, the implicit constants in big $O$ and little $o$ terms are universal as $\ep_1\rightarrow 0$, i.e.~they do not depend on any of $\ep_2, \alpha,\beta,\gamma,\delta$ (but may depend on $\phi$ and its derivatives). For example, a term is $O(\beta\ep_2)$ if it is bounded by $C\ep_2\beta$ as $\ep_1\rightarrow 0$ for some universal constant $C$. \medskip 

We first amend the computations in \eqref{11} to incorporate the new term involving $\gamma$. Note that at some places we simply write $o(1)$ for $\partial_a \phi$, but in others we use the fact that $\phi(0) = 0$ and $\nabla_{x'} \phi(0) = 0$ to assert more precisely that $\partial_a \phi = O(\ep_1^2) = O(\gamma^{-1}\ep_2)$ (for the last equality we have used \eqref{22}). With this in mind, for $a,b\in\{1,\dots,n-1\}$ we have
\begin{align}
\partial_a \xi_{\alpha,\beta,\gamma} & = \partial_a \ring{w} - \alpha \partial_a \phi - \beta(x_n - \phi)\partial_a \phi - 2\gamma x_a \nonumber \\
& = \partial_a \ring{w} + o(\alpha) + o(\beta\ep_2) + O(\gamma\ep_1)\nonumber \\
\partial_n \xi_{\alpha,\beta,\gamma} & = \alpha + \beta (x_n - \phi) \nonumber \\
& = \alpha + O(\beta\ep_2)\nonumber \\
\partial_{ab} \xi_{\alpha,\beta,\gamma} & = \partial_{ab} \ring{w} - \alpha \partial_{ab} \phi + \beta \partial_a \phi \partial_b \phi - \beta(x_n - \phi)\partial_{ab}\phi - 2\gamma\delta_{ab}\nonumber \\
& = \partial_{ab}\ring{w} - \alpha \partial_{ab}\phi - 2\gamma\delta_{ab} + o(\gamma^{-1}\beta\ep_2) + O(\beta\ep_2) \nonumber \\
\partial_{an} \xi_{\alpha,\beta,\gamma} & = -\beta \partial_a \phi \nonumber \\
& = O(\gamma^{-1}\beta\ep_2) \nonumber \\
\partial_{nn}\xi_{\alpha,\beta,\gamma} & = \beta.
\end{align}
Therefore we have
\begin{align}\label{61}
\nabla^2 \xi_{\alpha,\beta, \gamma} & = \left(
\begin{array}{c|c}
\partial_{ab}\ring{w} - \alpha \partial_{ab}\phi - 2\gamma\delta_{ab} + o(\gamma^{-1}\beta\ep_2) + O(\beta\ep_2) & O(\gamma^{-1}\beta\ep_2)\\
\hline
O(\gamma^{-1}\beta\ep_2)& \beta - 2\gamma
\end{array}
\right) \nonumber \\
& = \left(
\begin{array}{c|c}
\partial_{ab}\ring{w} - \alpha \partial_{ab}\phi - 2\gamma\delta_{ab} + O(\gamma^{-1}\beta\ep_2) & O(\gamma^{-1}\beta\ep_2)\\
\hline
O(\gamma^{-1}\beta\ep_2)& \beta - 2\gamma
\end{array}
\right),
\end{align}
where we have used the fact $\gamma<1$ to write $o(\gamma^{-1}\beta\ep_2) + O(\beta\ep_2) = O(\gamma^{-1}\beta\ep_2)$. We also have 
\begin{align}\label{62}
|\nabla \xi_{\alpha,\beta, \gamma}|^2  = |\nabla_\Sigma \ring{w}|^2 + \alpha^2 + E_2
\end{align}
where
\begin{align*}
E_2 = o(\alpha) + o(\beta\ep_2) + O(\gamma\ep_1) + \big[o(\alpha) + o(\beta\ep_2) + O(\gamma\ep_1)\big]^2 + O(\beta^2\ep_2^2).
\end{align*}
Finally, from the definition of $\xi_{\alpha,\beta, \gamma}$ and the bound $0<c\leq \delta\ep_2$ established in Claim 2, we have
\begin{align*}
\xi_{\alpha,\beta, \gamma} - c & = \ring{w}(x') - \gamma|x'|^2 + \alpha(x_n - \phi(x')) + \frac{1}{2}\beta(x_n - \phi(x'))^2 - c \nonumber \\
& =   \ring{w} + E_3
\end{align*}
where
\begin{align*}
E_3 = O(\gamma\ep_1^2) + O(\alpha\ep_2) + O(\beta\ep_2^2) + O(\delta\ep_2). 
\end{align*}
Combining \eqref{60}, \eqref{61} and \eqref{62}, we therefore have
\begin{align}\label{25}
- A_{\xi_{\alpha,\beta, \gamma} - c} & =  -(\xi_{\alpha,\beta, \gamma} - c)\nabla^2 \xi_{\alpha,\beta, \gamma}+ \frac{1}{2}|\nabla\xi_{\alpha,\beta, \gamma}|^2I\nonumber \\
& = -(\ring{w}+ E_3)\left(
\begin{array}{c|c}
\partial_{ab}\ring{w} - \alpha \partial_{ab}\phi - 2\gamma\delta_{ab} + O(\gamma^{-1}\beta\ep_2) & O(\gamma^{-1}\beta\ep_2)\\
\hline
O(\gamma^{-1}\beta\ep_2)& \beta - 2\gamma
\end{array}
\right) \nonumber \\
& \qquad + \frac{1}{2}\Big[|\nabla_\Sigma \ring{w}|^2 + \alpha^2 + E_2\Big]I.
\end{align}

Now, it is routine to see that for $\gamma>0$ fixed, all error terms in \eqref{25} \textit{with the possible exception of those involving a factor of $\beta\ep_2$} tend to zero as $\ep_1\rightarrow 0$. To deal with terms involving a factor of $\beta \ep_2$, we recall from \eqref{18} that $\beta \ep_2 = -3\delta$, and hence (by the definition of $\delta$ in \eqref{64}) such error terms can be made as small as desired by taking both $\ep_1$ small \textit{and} $\alpha$ close to $\partial_n w^+(0)$.\medskip 

After combining everything in \eqref{25} into one matrix, the entry in the bottom right corner is equal to $-(\ring{w}+E_3)(\beta -2\gamma) + \frac{1}{2}[|\nabla_\Sigma \ring{w}|^2 + \alpha^2 + E_2]$. By the reasoning in the previous paragraph, for $\ep_1$  sufficiently small and $\alpha$ is sufficiently close to $\partial_n w^+(0)$ this entry is positive at $x_1$. Recalling the differential inclusion $\lambda(-A_{\xi_{\alpha,\beta, \gamma} - c} (x_1))\in\Gamma_k^+$ established in Claim 2, it then follows from Cauchy's interlacing theorem that the top left $(n-1)\times(n-1)$ sub-matrix belongs to $\Gamma_{k-1}^+$:
\begin{align*}
\lambda\bigg((-\ring{w}+E_3)\big(\partial_{ab}\ring{w} - \alpha \partial_{ab} \phi - 2\gamma\delta_{ab} + O(\gamma^{-1}\beta\ep_2)\big) + \frac{1}{2}\Big[|\nabla_\Sigma \ring{w}|^2 + \alpha^2 + E_2\Big]\delta_{ab}\bigg) \in \Gamma_{k-1}^+
\end{align*}
at $x_1$. Still keeping $\gamma>0$ fixed, we now take $\alpha\rightarrow \partial_n w^+(0) \eqdef \alpha_*$ to obtain at $x_1$
\begin{align*}
\lambda\bigg((-\ring{w}+E_3)\big(\partial_{ab}\ring{w} - \alpha_* \partial_{ab} \phi - 2\gamma\delta_{ab} + O(\gamma^{-1}\beta\ep_2)\big) + \frac{1}{2}\Big[|\nabla_\Sigma \ring{w}|^2 + \alpha_*^2 + E_2\Big]\delta_{ab}\bigg) \in\overline{\Gamma_{k-1}^+}.
\end{align*}
We now take $\ep_1\rightarrow 0$, observing that in this limit, $x_1\rightarrow (0',0)$, $\delta\rightarrow 0$ (since we have already taken $\alpha\rightarrow \partial_n w^+(0)$) and hence $\beta\ep_2\rightarrow 0$. Therefore $E_2,E_3\rightarrow 0$, and hence at the origin we have 
\begin{align*}
\lambda\bigg(-\ring{w}\big(\partial_{ab}\ring{w} - \alpha_* \partial_{ab} \phi - 2\gamma\delta_{ab} \big) + \frac{1}{2}\Big[|\nabla_\Sigma \ring{w}|^2 + \alpha_*^2 \Big]\delta_{ab}\bigg) \in\overline{\Gamma_{k-1}^+}.
\end{align*}
Since $\gamma>0$ was arbitrary, we may finally take $\gamma\rightarrow 0$ to obtain at the origin
\begin{align*}
\lambda\bigg(-\ring{w}\big(\partial_{ab}\ring{w} - \alpha_* \partial_{ab} \phi\big) + \frac{1}{2}\Big[|\nabla_\Sigma \ring{w}|^2 + \alpha_*^2\Big]\delta_{ab}\bigg) \in\overline{\Gamma_{k-1}^+},
\end{align*}
as required.
\end{proof} 

\begin{proof}[Proof of Theorem \ref{A}]
	Theorem \ref{A} follows immediately from Lemmas \ref{24} and \ref{21}. 
\end{proof}

\subsection{Proof of Theorem \ref{D}}\label{s4}

\begin{proof}[Proof of Theorem \ref{D}]
	By Lemma \ref{4}, $\nabla_\nu w^+$ is the smaller of the two roots $\nabla_\nu w^{\pm}$ to \eqref{32}. Since the quadratic in \eqref{32} has positive leading order coefficient, the derivative of the LHS of \eqref{32} is therefore negative at $\alpha = \nabla_\nu^+ w$, i.e.~
	\begin{align*}
	w_0 H_\Sigma + (n-1) \nabla_\nu w^+ < 0 \quad \text{on }\Sigma\cap B_\ep^+. 
	\end{align*}
	As explained above Proposition \ref{D} in the introduction, this is precisely the assertion that the mean curvature of $\Sigma\cap B_\ep$ is negative with respect to $g|_{\overline{B_\ep^+}} = w_+^{-2}|dx|^2$. The proof is similar in $B_\ep^-$. 
\end{proof}

\section{Proof of Proposition \ref{B}}\label{s3}

In this section we prove Proposition \ref{B}. We fix a point $x_0$ and the associated coordinate system as described previously, and we always implicitly work sufficiently close to $\Sigma$ so that $\pi(x) = x_0$ for $x$ on the normal line to $\Sigma$ at $x_0$, where $\pi$ denotes the closest point projection onto $\Sigma$.

\begin{proof} 
We consider only the case $B_\ep^+$ (the other side is almost identical) and for ease of notation we drop any $+$ superscripts. Thus for $p\in (1,2)$ and $w_*$ not yet fixed, we denote
\begin{align*}
\bar{w}(x) = \widetilde{w}_0(\pi(x)) + w_1(\pi(x)) \mathrm{d}(x) + w_*(\pi(x)) \mathrm{d}(x)^p.
\end{align*}
Since $A_{\bar{w}} = \bar{w}\nabla^2\bar{w} - \frac{1}{2}|\nabla \bar{w}|^2 I$ and $\sigma_2(\lambda(-A_{\bar{w}})) = \frac{1}{2}(\operatorname{tr}(A_{\bar{w}}) - |A_{\bar{w}}|^2)$ we have
\begin{align}\label{30}
\sigma_2(\lambda(-A_{\bar{w}})) & = \frac{1}{2}\bigg(\Big(\bar{w}\Delta \bar{w} - \frac{n}{2}|\nabla \bar{w}|^2\Big)^2 - \Big\langle \bar{w}\nabla^2 \bar{w} - \frac{1}{2}|\nabla \bar{w}|^2 I, \bar{w}\nabla^2 \bar{w} - \frac{1}{2}|\nabla \bar{w}|^2 I\Big\rangle\bigg) \nonumber \\
& = \frac{1}{2}\bigg(\bar{w}^2\big((\Delta \bar{w})^2 - |\nabla^2 \bar{w}|^2\big) - (n-1)\bar{w}|\nabla \bar{w}|^2 \Delta \bar{w} + \frac{n(n-1)}{4}|\nabla \bar{w}|^4\bigg). 
\end{align}
In what follows, we always compute implicitly at a point $\bar{x} = (0', x_n)$. \medskip

\noindent\textbf{Step 1:} In this step we compute the quantities $|\nabla \bar{w}|^2$, $|\nabla^2 \bar{w}|^2$ and $(\Delta\bar{w})^2$, but postpone until Step 2 the expansion of terms involving purely tangential derivatives of $\bar{w}$. We first compute
\begin{align*}
\nabla \bar{w} &  = (\nabla_T \bar{w}, \partial_{n}\bar{w}) = (\nabla_T \bar{w}, w_1 + pw_* d^{p-1}),
\end{align*}
which implies 
\begin{align}\label{80}
|\nabla \bar{w}|^2 & = |\nabla_T \bar{w}|^2 + (\partial_{n}\bar{w})^2= |\nabla_T \bar{w}|^2 + w_1^2 + 2pw_1w_* d^{p-1}  + p^2 w_*^2 d^{2(p-1)} +  O(d).
\end{align}
In what follows, a subscript $T$ signifies we are computing only in the tangential coordinates (these do not generally coincide with quantities defined with respect to the induced metric on $\Sigma$ except at 0). To compute 
\begin{align*}
|\nabla^2 \bar{w}|^2 = |\nabla^2_T \bar{w}|^2 + 2|\nabla_T \partial_n \bar{w}|^2 + (\partial_{nn}^2 \bar{w})^2 \quad \text{and}\quad (\Delta\bar{w})^2 = (\Delta_T \bar{w})^2 + 2\Delta_T \bar{w}\, \partial_{nn}^2 \bar{w} + (\partial_{nn}^2 \bar{w})^2,
\end{align*}
observe that
\begin{align*}
\partial_{nn}^2 \bar{w} = p(p-1)w_*d^{p-2} 
\end{align*}
and
\begin{align*}
\nabla_T \partial_{n}\bar{w}  = \nabla_T(w_1 + pw_* d^{p-1}) & = \nabla_T w_1 + pd^{p-1}\nabla_T w_* + p(p-1)w_* d^{p-2}\nabla_T d \nonumber \\
& = \nabla_T w_1 + pd^{p-1}\nabla_T w_* ,
\end{align*}
where we have used $\nabla_T d = 0$ at $\bar{x}$ to reach the last line. It follows that
\begin{align}\label{81}
|\nabla^2 \bar{w}|^2 & = |\nabla^2_T \bar{w}|^2 + 2|\nabla_T w_1|^2+ 4pd^{p-1} \langle\nabla_T w_1, \nabla_T w_*\rangle + 2p^2 d^{2(p-1)}|\nabla_T w_*|^2 + (\partial_{nn}^2 \bar{w})^2 \nonumber \\
& = |\nabla^2_T \bar{w}|^2 + 2|\nabla_T w_1|^2 + (\partial_{nn}^2 \bar{w})^2 + O(d^{p-1}). 
\end{align}
and
\begin{align}\label{82}
(\Delta \bar{w})^2 = (\Delta_T\bar{w})^2 + 2p(p-1)w_* d^{p-2}\Delta_T \bar{w} + (\partial_{nn}^2 \bar{w})^2.
\end{align}

\noindent\textbf{Step 2:} In this step we first expand out the purely tangential derivatives of $\bar{w}$ appearing in \eqref{80}, \eqref{81} and \eqref{82}, and then piece these computations together to obtain expressions for $\bar{w}^2\big((\Delta \bar{w})^2 - |\nabla^2 \bar{w}|^2\big)$, $\bar{w}|\nabla \bar{w}|^2 \Delta \bar{w}$ and $|\nabla \bar{w}|^4$ (which constitute the right hand side of \eqref{30}).\medskip 

Using once again the fact that $\nabla_T d = 0$ at $\bar{x}$, we first compute
\begin{align}\label{83}
\nabla_T \bar{w} & = \nabla_T \widetilde{w}_0 + d\nabla_T w_1 + d^p \nabla_T w_* + (w_1 + pd^{p-1}w_* )\nabla_T d \nonumber \\
& = \nabla_T \widetilde{w}_0 + d\nabla_T w_1 + d^p \nabla_T w_*,
\end{align}
which implies
\begin{align}\label{87}
|\nabla_T \bar{w}|^2 = |\nabla_T \widetilde{w}_0|^2 + O(d). 
\end{align}
Using \eqref{83} we also see
\begin{align*}
\nabla_T^2 \bar{w} & = \nabla_T^2 \widetilde{w}_0 + d\nabla_T^2 w_1 + d^p \nabla_T^2 w_* + (w_1 + pd^{p-1}w_* )\nabla_T^2 d \nonumber \\
& =  \nabla_T^2 \widetilde{w}_0 + d\nabla_T^2 w_1 + d^p \nabla_T^2 w_* + (w_1 + pd^{p-1}w_*)(-\nabla_T^2 \phi + O(d)),
\end{align*}
where the $O(d)$ term in the final line follows from the fact that the hypersurfaces parallel to $\Sigma$ vary smoothly in a sufficiently small normal neighbourhood of $\Sigma$. It follows that
\begin{align}\label{84}
\Delta_T \bar{w}  = \Delta_T \widetilde{w}_0 - w_1\Delta_T \phi  - pd^{p-1}w_*\Delta_T \phi + O(d), 
\end{align}
\begin{align}\label{85}
(\Delta_T \bar{w})^2 = (\Delta_T \widetilde{w}_0 - w_1\Delta_T \phi)^2 + O(d^{p-1})
\end{align}
and
\begin{align}\label{86}
|\nabla_T^2 \bar{w}|^2 = |\nabla^2_T \widetilde{w}_0 - w_1\nabla^2_T \phi|^2 + O(d^{p-1}). 
\end{align}

Substituting \eqref{87} into \eqref{80}, \eqref{84} and \eqref{85} into \eqref{82}, and \eqref{86} into \eqref{81}, we therefore see that
\begin{align}\label{29}
 \bar{w}^2\big((\Delta \bar{w})^2 - |\nabla^2 \bar{w}|^2\big) &  = \big(\widetilde{w}_0^2 + O(d) \big) \bigg(2p(p-1)w_* d^{p-2} \big[\Delta_T \widetilde{w}_0 - w_1\Delta_T \phi  - pd^{p-1}w_*\Delta_T \phi \big]  \nonumber \\
& \quad + (\Delta_T \widetilde{w}_0 - w_1\Delta_T \phi)^2   - |\nabla^2_T \widetilde{w}_0 - w_1\nabla^2_T \phi|^2  - 2|\nabla_T w_1|^2  + O(d^{p-1})\bigg),
\end{align}
\begin{align}\label{28}
\bar{w}|\nabla \bar{w}|^2 \Delta \bar{w} & = \big(\widetilde{w}_0+ O(d) \big)\bigg(|\nabla_T \widetilde{w}_0|^2 + w_1^2 + 2pw_1w_* d^{p-1} + p^2 w_*^2 d^{2(p-1)}+ O(d)\bigg)\cdot \nonumber \\
& \quad \, \bigg(p(p-1)w_*d^{p-2}  + \Delta_T \widetilde{w}_0 - w_1\Delta_T \phi  - pd^{p-1}w_*\Delta_T \phi +  O(d)\bigg)
\end{align}
and
\begin{align}\label{39}
|\nabla \bar{w}|^4 & = \bigg(|\nabla_T \widetilde{w}_0|^2 + w_1^2 + O(d^{p-1})\bigg)^2.
\end{align}

\noindent\textbf{Step 3:} In this step we use \eqref{29}--\eqref{39} to collect terms appearing in \eqref{30} of the same order. We also show that when $p=\frac{3}{2}$, we may uniquely (and explicitly) determine $w_*$ such that \eqref{69} is satisfied if and only if \eqref{65'} holds. \medskip 

The lowest order expression in $d$ appearing in \eqref{29}--\eqref{39} is $d^{p-2}$. Such expressions only arise in \eqref{29} and \eqref{28}, and thus after recalling \eqref{30} we see by inspection that corresponding terms in $\sigma_2(\lambda(-A_{\bar{w}}))$ are 
\begin{align*}
p(p-1)\widetilde{w}_0w_*d^{p-2}\bigg(\widetilde{w}_0\Delta_T \widetilde{w}_0 - \widetilde{w}_0w_1\Delta_T \phi - \frac{n-1}{2} (|\nabla_T \widetilde{w}_0|^2 + w_1^2)\bigg),
\end{align*}
which is equal to zero by \eqref{88}.\medskip

Let us now take $p=\frac{3}{2}$, so that $d^{2p-3} = 1$. Then following $d^{p-2}$, the next lowest order terms are those which are independent of $d$, namely $d^{2p-3}$ \textit{and} those terms that were already independent of $d$ before taking $p=\frac{3}{2}$. In \eqref{29}, these terms are collectively 
\begin{align*}
& \widetilde{w}_0^2\Big(\!\!-\!2p(p-1)d^{p-2}w_*\!\cdot\! pd^{p-1}w_* \Delta_T \phi + (\Delta_T \widetilde{w}_0 - w_1\Delta_T \phi)^2   - |\nabla_T^2\widetilde{w}_0 - w_1\nabla^2_T \phi|^2 - 2|\nabla_T w_1|^2  \Big), 
\end{align*}
in \eqref{28} they are
\begin{align*}
\widetilde{w}_0 \cdot 2pw_1 w_* d^{p-1}\cdot p(p-1)w_* d^{p-2} + \widetilde{w}_0(|\nabla_T \widetilde{w}_0|^2 + w_1^2)(\Delta_T \widetilde{w}_0 - w_1\Delta_T \phi ), 
\end{align*}
and in \eqref{39} they are $(|\nabla_T \widetilde{w}_0|^2 + w_1^2)^2$. Thus, the corresponding terms in $\sigma_2(\lambda(-A_{\bar{w}}))$ are
\begin{align}\label{41}
\frac{1}{2}\bigg[&\widetilde{w}_0^2\Big((\Delta_T \widetilde{w}_0 - w_1\Delta_T \phi)^2 - |\nabla_T^2 \widetilde{w}_0 - w_1\nabla^2_T \phi|^2 - 2|\nabla_T w_1|^2 - 2p^2(p-1)w_*^2 \Delta_T \phi \Big) \nonumber \\
&  -(n-1)\Big(\widetilde{w}_0(|\nabla_T \widetilde{w}_0|^2 + w_1^2)(\Delta_T \widetilde{w}_0 - w_1\Delta_T \phi )  + 2p^2(p-1)\widetilde{w}_0 w_1 w_*^2\Big) \nonumber \\
& + \frac{n(n-1)}{4}\big(|\nabla_T\widetilde{w}_0|^2 + w_1^2\big)^2\bigg]. 
\end{align}
Since the quantity in \eqref{41} is precisely $\lim_{d\rightarrow 0}\sigma_2(\lambda(-A_{\bar{w}}))$, we wish to set it equal to 1 and solve for $w_*$. Denoting $X = \nabla^2_T \widetilde{w}_0- w_1\nabla^2_T \phi$, the equation reads
\begin{align}\label{42}
1 & = \frac{1}{2}\widetilde{w}_0^2\Big(\operatorname{tr}(X)^2 - |X|^2 - 2|\nabla_T w_1|^2 - 2p^2(p-1)w_*^2\Delta_T \phi\Big) -\frac{n-1}{2}\widetilde{w}_0(|\nabla_T \widetilde{w}_0|^2 + w_1^2)\operatorname{tr}(X)  \nonumber \\
& \quad - (n-1)p^2(p-1)\widetilde{w}_0 w_1 w_*^2  + \frac{n(n-1)}{8}\big(|\nabla_T\widetilde{w}_0|^2 + w_1^2\big)^2.
\end{align}
Next observe by \eqref{32} that $|\nabla_T \widetilde{w}_0|^2 + w_1^2 = \frac{2}{n-1}\widetilde{w}_0\operatorname{tr}(X)$. Substituting this into \eqref{42} and simplifying we get
\begin{align}\label{43}
1 & = \frac{1}{2}\widetilde{w}_0^2\bigg(\frac{1}{n-1}\operatorname{tr}(X)^2 - |X|^2 - 2|\nabla_T w_1|^2\bigg) - p^2(p-1)\widetilde{w}_0 w_*^2\Big(\widetilde{w}_0\Delta_T \phi + (n-1)w_1\Big) \nonumber \\
& = \frac{1}{2}\widetilde{w}_0^2\bigg(-|\mathring{X}|^2 - 2|\nabla_T w_1|^2\bigg) - p^2(p-1)\widetilde{w}_0 w_*^2\Big(\widetilde{w}_0\Delta_T \phi + (n-1)w_1\Big),
\end{align}
where $\mathring{X}$ is the trace-free part of $X$. This has real solutions for $w_*$ if and only if
\begin{align*}
\widetilde{w}_0\Delta_T \phi + (n-1)w_1 < 0,
\end{align*}
which is precisely \eqref{65'} since, in our principal coordinate system, $\Delta_T\phi(0)$ is the mean curvature of $\Sigma$ at the origin.\medskip

To uniquely determine $w_*$, we must determine which root to take in the equation \eqref{43} for $w_*^2$ in the case that it admits real solutions. For this we use the fact that $\lambda(-A_{\bar{w}})\in\Gamma_1^+$ near $\Sigma$:
\begin{align*}
0  < \sigma_1(\lambda(-A_{\bar{w}})) = -\bar{w}\Delta \bar{w} + \frac{n}{2}|\nabla \bar{w}|^2  = -\big(\widetilde{w}_0 + O(d)\big)\big(p(p-1)w_*d^{p-2} + O(d)\big) + O(1).
\end{align*}
As $d\rightarrow 0$ the term $-p(p-1)\widetilde{w}_0 w_*d^{p-2}$ dominates and therefore $-p(p-1)\widetilde{w}_0w_*$ must be positive. This forces $w_*<0$, i.e.~we take the negative root in \eqref{43}. We have therefore shown that if $p=\frac{3}{2}$, then $w_*$ exists if and only if \eqref{65'} holds, and moreover $w_*$ is negative and unique. In particular, we have proved the backward implication stated in Proposition \ref{B}. \medskip 

\noindent\textbf{Step 4:} In this step we complete the proof of Proposition \ref{B} by proving the forward implication. Let us suppose first for a contradiction that $p<\frac{3}{2}$ and there exists $w_* \not \equiv 0$ such that \eqref{69} is satisfied. Since $d^{2p-3}\rightarrow+\infty$ in this instance, \eqref{69} implies that the coefficient of $d^{2p-3}$ in $\sigma_2(\lambda(-A_{\hat{w}}))$ must vanish. By the computations in Step 3, this coefficient is 
\begin{align*}
-p^2(p-1)\widetilde{w}_0w_*^2\big(\widetilde{w}_0\Delta_T \phi + (n-1)w_1\big). 
\end{align*}
Since $w_* \not\equiv 0$, it must therefore be the case that $\widetilde{w}_0\Delta_T \phi + (n-1)w_1$ vanishes at some point in $\Sigma\cap B_\ep$, contradicting the inequality \eqref{65'}.\medskip

On the other hand, by inspecting \eqref{29}--\eqref{39} we see there is no value of $p\in (\frac{3}{2},2)$ which would introduce terms independent of $d$ in $\sigma_2(\lambda(-A_{\bar{w}}))$. Then the computations in Step 3 show that the only terms independent of $d$ in $\sigma_2(\lambda(-A_{\bar{w}}))$ are $\frac{1}{2}\widetilde{w}_0^2(-|\mathring{X}|^2 - 2|\nabla_T w_1|^2)$, and thus the condition \eqref{69} implies 
\begin{align*}
1 = \frac{1}{2}\widetilde{w}_0^2\bigg(-|\mathring{X}|^2 - 2|\nabla_T w_1|^2\bigg),
\end{align*}
which is clearly not possible. 
\end{proof}


\footnotesize
\begin{thebibliography}{10}

\bibitem{AILA18}
{\sc P.~T. Allen, J.~Isenberg, J.~M. Lee, and I.~S. Allen}, {\em Weakly
  asymptotically hyperbolic manifolds}, Comm. Anal. Geom., 26 (2018),
  pp.~1--61.

\bibitem{ACF92}
{\sc L.~Andersson, P.~T. Chru\'{s}ciel, and H.~Friedrich}, {\em On the
  regularity of solutions to the {Y}amabe equation and the existence of smooth
  hyperboloidal initial data for {E}instein's field equations}, Comm. Math.
  Phys., 149 (1992), pp.~587--612.

\bibitem{Av82}
{\sc P.~Aviles}, {\em A study of the singularities of solutions of a class of
  nonlinear elliptic partial differential equations}, Comm. Partial
  Differential Equations, 7 (1982), pp.~609--643.

\bibitem{AM88}
{\sc P.~Aviles and R.~C. McOwen}, {\em Complete conformal metrics with negative
  scalar curvature in compact {R}iemannian manifolds}, Duke Math. J., 56
  (1988), pp.~395--398.

\bibitem{CLL23}
{\sc B.~Z. Chu, Y.~Y. Li, and Z.~Li}, {\em Liouville theorems for conformally
  invariant fully nonlinear equations {I}}, arXiv:2311.07542 [math.AP],
  (2023).

\bibitem{DN23}
{\sc J.~A.~J. Duncan and L.~Nguyen}, {\em The $\sigma_k$-{L}oewner--{N}irenberg
  problem on {R}iemannian manifolds for $k<\frac{n}{2}$}, Anal. PDE, 18 (2025),
  pp.~2203--2240.

\bibitem{DN25b}
\leavevmode\vrule height 2pt depth -1.6pt width 23pt, {\em The
  {$\sigma_k$}-{L}oewner-{N}irenberg problem on {R}iemannian manifolds for
  {$k=\frac{n}{2}$} and beyond}, J. Funct. Anal., 290 (2026), pp.~Paper No.
  111306, 36.

\bibitem{DN25a}
\leavevmode\vrule height 2pt depth -1.6pt width 23pt, {\em The fully nonlinear
  {L}oewner-{N}irenberg problem: {L}iouville theorems and counterexamples to
  local boundary estimates}, https://arxiv.org/abs/2507.16383,  (submitted
  2025).

\bibitem{FW20}
{\sc H.~Fang and W.~Wei}, {\em A {$\sigma_2$} {P}enrose inequality for
  conformal asymptotically hyperbolic 4-discs}, Adv. Math., 402 (2022),
  pp.~Paper No. 108365, 33.

\bibitem{Finn98}
{\sc D.~L. Finn}, {\em Existence of positive solutions to {$\Delta_gu=u^q+Su$}
  with prescribed singularities and their geometric implications}, Comm.
  Partial Differential Equations, 23 (1998), pp.~1795--1814.

\bibitem{GLN18}
{\sc M.~d.~M. Gonz\'{a}lez, Y.~Y. Li, and L.~Nguyen}, {\em Existence and
  uniqueness to a fully nonlinear version of the {L}oewner-{N}irenberg
  problem}, Commun. Math. Stat., 6 (2018), pp.~269--288.

\bibitem{GW17}
{\sc A.~R. Gover and A.~Waldron}, {\em Renormalized volume}, Comm. Math. Phys.,
  354 (2017), pp.~1205--1244.

\bibitem{Gr17}
{\sc C.~R. Graham}, {\em Volume renormalization for singular {Y}amabe metrics},
  Proc. Amer. Math. Soc., 145 (2017), pp.~1781--1792.

\bibitem{GG21}
{\sc C.~R. Graham and M.~J. Gursky}, {\em Chern-{G}auss-{B}onnet formula for
  singular {Y}amabe metrics in dimension four}, Indiana Univ. Math. J., 70
  (2021), pp.~1131--1166.

\bibitem{Guan08}
{\sc B.~Guan}, {\em Complete conformal metrics of negative {R}icci curvature on
  compact manifolds with boundary}, Int. Math. Res. Not. IMRN,  (2008).
\newblock Art. ID rnn 105, 25pp.

\bibitem{GSW11}
{\sc M.~J. Gursky, J.~Streets, and M.~Warren}, {\em Existence of complete
  conformal metrics of negative {R}icci curvature on manifolds with boundary},
  Calc. Var. PDE, 41 (2011), pp.~21--43.

\bibitem{GV03b}
{\sc M.~J. Gursky and J.~A. Viaclovsky}, {\em Fully nonlinear equations on
  {R}iemannian manifolds with negative curvature}, Indiana Univ. Math. J., 52
  (2003), pp.~399--419.

\bibitem{HJS20}
{\sc Q.~Han, X.~Jiang, and W.~Shen}, {\em The {L}oewner-{N}irenberg problem in
  cones}, J. Funct. Anal., 287 (2024), p.~110566.

\bibitem{HS20}
{\sc Q.~Han and W.~Shen}, {\em The {L}oewner-{N}irenberg problem in singular
  domains}, J. Funct. Anal., 279 (2020), pp.~108604, 43.

\bibitem{HN23}
{\sc J.~Hogg and L.~Nguyen}, {\em Existence and uniqueness for the non-compact
  {Y}amabe problem of negative curvature type}, Anal. Theory Appl., 40 (2024),
  pp.~57--91.

\bibitem{Jia21}
{\sc X.~Jiang}, {\em Boundary expansion for the {L}oewner-{N}irenberg problem
  in domains with conic singularities}, J. Funct. Anal., 281 (2021), pp.~Paper
  No. 109122, 41.

\bibitem{Li22}
{\sc G.~Li}, {\em Two flow approaches to the {L}oewner-{N}irenberg problem on
  manifolds}, J. Geom. Anal., 32 (2022), pp.~Paper No. 7, 30.

\bibitem{LN20b}
{\sc Y.~Y. Li and L.~Nguyen}, {\em Solutions to the
  {$\sigma_k$}-{L}oewner-{N}irenberg problem on annuli are locally {L}ipschitz
  and not differentiable}, J. Math. Study, 54 (2021), pp.~123--141.

\bibitem{LN20}
\leavevmode\vrule height 2pt depth -1.6pt width 23pt, {\em Existence and
  uniqueness of {G}reen's functions to nonlinear {Y}amabe problems}, Comm. Pure
  Appl. Math., 76 (2023), pp.~1554--1607.

\bibitem{LNW18}
{\sc Y.~Y. Li, L.~Nguyen, and B.~Wang}, {\em Comparison principles and
  {L}ipschitz regularity for some nonlinear degenerate elliptic equations},
  Calc. Var. Partial Differential Equations, 57 (2018), pp.~Paper No. 96, 29.

\bibitem{LNX22}
{\sc Y.~Y. Li, L.~Nguyen, and J.~Xiong}, {\em Regularity of viscosity solutions
  of the {$\sigma_k$}-{L}oewner-{N}irenberg problem}, Proc. Lond. Math. Soc.
  (3), 127 (2023), pp.~1--34.

\bibitem{LN74}
{\sc C.~Loewner and L.~Nirenberg}, {\em Partial differential equations
  invariant under conformal or projective transformations}, in Contributions to
  analysis (a collection of papers dedicated to {L}ipman {B}ers), Academic
  Press, New York-London, 1974, pp.~245--272.

\bibitem{Maz91}
{\sc R.~Mazzeo}, {\em Regularity for the singular {Y}amabe problem}, Indiana
  Univ. Math. J., 40 (1991), pp.~1277--1299.

\bibitem{MP03}
{\sc R.~Mazzeo and F.~Pacard}, {\em Poincar\'{e}-{E}instein metrics and the
  {S}chouten tensor}, Pacific J. Math., 212 (2003), pp.~169--185.

\bibitem{Sav07}
{\sc O.~Savin}, {\em Small perturbation solutions for elliptic equations},
  Comm. Partial Differential Equations, 32 (2007), pp.~557--578.

\bibitem{Ver81}
{\sc L.~V\'{e}ron}, {\em Singularit\'{e}s \'{e}liminables d'\'{e}quations
  elliptiques non lin\'{e}aires}, J. Differential Equations, 41 (1981),
  pp.~87--95.

\bibitem{Wu24}
{\sc J.~Wu}, {\em Regularity of viscosity solutions of the
  $\sigma_k$-{Y}amabe-type {P}roblem for $k>n/2$},
  https://arxiv.org/abs/2407.08300,  (2024).

\bibitem{Yuan22}
{\sc R.~Yuan}, {\em The partial uniform ellipticity and prescribed problems on
  the conformal classes of complete metrics}, https://arxiv.org/abs/2203.13212,
   (2022).

\bibitem{Yuan24}
\leavevmode\vrule height 2pt depth -1.6pt width 23pt, {\em Notes on conformal
  metrics of negative curvature on manifolds with boundary}, J. Math. Study, 57
  (2024), pp.~373--378.

\end{thebibliography}
\end{document}